\documentclass[a4paper,12pt]{amsart}
\usepackage[english]{babel}
\usepackage[utf8]{inputenc}
\usepackage[T1]{fontenc}

\usepackage[style=alphabetic,useprefix,hyperref,backend=bibtex]{biblatex}
\bibliography{./HorofunctionsInHeisebbergBIB}

\usepackage{amssymb}
\usepackage{mathrsfs}
\usepackage{hyperref}
\usepackage[usenames,dvipsnames]{xcolor}
\hypersetup{colorlinks,%
citecolor=OliveGreen,%
filecolor=magenta,%
linkcolor=RoyalBlue,%
urlcolor=cyan}
\usepackage{csquotes}
\usepackage{mparhack}	
\usepackage{comment}	
\usepackage{marvosym}	
\usepackage{enumitem}	
\usepackage{array}		
\usepackage[pdftex]{graphicx}	

\usepackage{cancel}

\newcommand{\scr}[1]{\mathscr{#1}}
\newcommand{\frk}[1]{\mathfrak{#1}}
\newcommand{\bb}[1]{\mathbb{#1}}

\renewcommand{\rm}[1]{\mathrm{#1}}

\newcommand{\N}{\mathbb{N}}	
\newcommand{\Z}{\mathbb{Z}}	
\newcommand{\R}{\mathbb{R}}	
\newcommand{\Co}{\mathscr{C}}	
\newcommand{\dd}{\,\mathrm{d}}	
\newcommand{\de}{\partial}		
\newcommand{\THEN}{\Rightarrow}	
\newcommand{\IFF}{\Leftrightarrow}	

\newcommand{\into}{\hookrightarrow}		

\newcommand{\ci}{\framebox[\width]{$\subset$} }
\newcommand{\ic}{\framebox[\width]{$\supset$} }

\newcommand{\dcc}{d_{CC}}		
\newcommand{\dr}{d_R}	

\newcommand{\HH}{\bb H}	
\newcommand{\hh}{\frk h}	
\newcommand{\hhh}{\hh/[\hh,\hh]} %
\newcommand{\BA}{\scr A}
\newcommand{\mc}{\omega_\HH} 

\theoremstyle{plain}
\newtheorem{Prop}{Proposition}[section]
\newtheorem{Teo}[Prop]{Theorem}
\newtheorem{Lem}[Prop]{Lemma}
\newtheorem{Cor}[Prop]{Corollary}
\newtheorem{Cor*}{Corollary}

\theoremstyle{definition}
\newtheorem{Def}[Prop]{Definition}
\newtheorem{Rem}[Prop]{Remark}

\title
[Asymptotic behavior and horoboundary of \texorpdfstring{$\HH$}{H}]
{Asymptotic behavior of the Riemannian Heisenberg group and its horoboundary}
\author[Le Donne]{Enrico Le Donne}
\author[Nicolussi Golo]{Sebastiano Nicolussi Golo}
\thanks{S.N.G.~is supported by a doctoral fellowship in the project MaNeT}

\address[Le Donne, Nicolussi Golo]{University of Jyvaskyla, Finland}%

\author[Sambusetti]{Andrea Sambusetti}
\address[Sambusetti]{Sapienza Universit\`a di Roma, Italy}

\date{September 1, 2015}
\keywords{Heisenberg group, horoboundary, asymptotic cone, Riemannian geometry, sub-Riemannian geometry}
\subjclass[2010]{20F69, 53C23, 53C17}

\begin{document}
\maketitle

\begin{abstract}
 The paper is devoted to the large scale geometry of the Heisenberg group $\HH$ equipped with left-invariant Riemannian distances.
 We prove that two such distances have bounded difference if and only if they are asymptotic, i.e., their ratio goes to one, at infinity.
 Moreover, we show that for every left-invariant Riemannian distance $d$ on $\HH$ there is a unique subRiemanniann metric $d’$ for which $d-d’$ goes  to zero at infinity, and we estimate the rate of convergence. 
 As a first immediate consequence we get that the Riemannian Heisenberg group is at bounded distance from its asymptotic cone.
 The second consequence, which was our aim, is the explicit description of the horoboundary of the Riemannian Heisenberg group.
\end{abstract}

\setcounter{tocdepth}{1}
\tableofcontents


\section{Introduction}
In large-scale geometry, various notions of a space at infinity have received special interest for differently capturing the asymptotic geometric behavior. 
%
%
%
 Two main examples of spaces at infinity are the asymptotic cone  and the horoboundary.
The description of the 
\emph{asymptotic cone} for finitely generated groups 
is a crucial step 
in the algebraic characterization of groups of polynomial growth, \cite{MR0232311,MR0286898,MR623534,MR0248688,MR741395,MR3153949}.
The notion of \emph{horoboundary}  has been formulated 
by Gromov \cite{MR624814}, inspired by the seminal work of  Busemann on the theory of parallels on geodesic spaces 
\cite{MR0075623}.
The horoboundary  has a fully satisfying visual description in the framework of $CAT (0)$-spaces and  
of Gromov-hyperbolic spaces, \cite{MR919829, MR1377265,  MR1744486}. 
It plays a major role in the study of dynamics and rigidity  of negatively curved spaces,   \cite{Hadamard1898Les-surfaces-a-, MR0385004, MR919829,MR556586, MR0450547, MR2057305}.
The  
visual-boundary description 
breaks down for non-simply connected manifolds  \cite{MR2988483} and when the curvature has variable sign, as we will make evident for the Riemannian Heisenberg group.

%

This paper contributes to the study of the asymptotic geometry of the simplest non-Abelian nilpotent group: the Heisenberg group. 
The asymptotic cone of the Heisenberg group equipped with a left-invariant Riemannian metric $\dr$ is the Heisenberg group equipped with a Carnot-Carathéodory distance $\dcc$, see \cite{MR741395} and also \cite{MR3153949}.
Our contribution is a finer analysis of the asymptotic comparison of $\dr$ and $\dcc$.
This leads to the explicit knowledge of the (Riemannian) horoboundary.
We remark that the Heisenberg group is not hyperbolic, hence one does not consider its visual boundary.



We recall the definition of horoboundary.
Let $(X,d)$ be a metric space. 
We consider the space of continuous real functions $\Co(X)$ endowed with the topology of uniform convergence on compact sets.
We denote by $\Co(X)/\R$ the quotient with respect to the subspace of constant functions.
The embedding $x\mapsto d(x,\cdot)$
induces an embedding $X\into\Co(X)/\R$.
The horoboundary of $X$ is defined as $\de_hX := \bar X\setminus X\subset\Co(X)/\R$.
See Section~\ref{sec24031114} for a detailed exposition.


The horoboundary of finite dimensional normed vector spaces has been 
investigated for normed spaces,  \cite{MR2329698}, 
for Hilbert geometries,  \cite{MR2456635},
and for infinite graphs 
\cite{MR2219016}.
For non-simply connected, negatively curved manifolds it has been studied in \cite{MR2988483}. 
Nicas and Klein computed the horoboundary of the Heisenberg group when endowed with the Korany metric in \cite{MR2546714}, and with the metric $\dcc$ in \cite{MR2738530}.

We will show that the horoboundary of the Heisenberg group endowed with a left-invariant Riemannian metric $\dr$ coincides with the second case studied by Nicas and Klein, see Corollary~\ref{TEO3}.
This will be an immediate consequence of our main result Theorem~\ref{TEO2}, which implies that the difference $\dr-\dcc$ converges to zero as the distances diverge.

\subsection{Detailed results}
The Heisenberg group $\HH$ is the simply connected Lie group whose Lie algebra $\hh$ is generated by three vectors $X,Y,Z$ with only non-zero relation $[X,Y]=Z$.
%
A left-invariant Riemannian metric $d$ on $\HH$ is determined by a scalar product $g$ on $\hh$; a left-invariant \emph{strictly} subRiemannian metric $d$ is induced by a bracket generating plane $V\subset\hh$ and a scalar product $g$ on $V$ (see Section~\ref{sec14031114} for detailed exposition). 
In both cases we say that $d$ is subRiemannian with \emph{horizontal space} $(V,g)$, where $\dim V$ is either 2 or 3;
 subRiemannian metrics are also called Carnot-Carath\'eodory.

We are interested in the asymptotic comparison between these distances.
Given two left-invariant subRiemannian metrics $d$ and $d'$ on $\HH$, we deal with three asymptotic behaviors, in ascending order of strength,  each of which defines an equivalence relation among subRiemannian distances:

\begin{enumerate}[label=(\roman*)]
\item\label{A} 	$\lim_{d(p,q)\to\infty}\frac{d(p,q)}{d'(p,q)} = 1$;
\item\label{AA} 	There exists $c>0$ such that $|d(p,q)-d'(p,q)|< c$, for all $p,q$;
\item\label{AAA} 	$\lim_{d(p,q)\to\infty}|d(p,q)-d'(p,q)|=0$.
\end{enumerate}

A first example of the implication $\ref{A}\THEN\ref{AA}$ was proved by Burago in 
\cite{MR1279391} for $\Z^n$-invariant metrics $d$ on $\R^n$, by showing that $d$ and the associated stable norm stay at bounded distance from each other.
This result has been extended quantitatively 
  for $\Z^n$-invariant metrics on geodesic metric spaces in \cite{2014arXiv1412.6516C}.
 Gromov and  Burago   asked for other interesting cases where the same implication holds. 
Another well-known case where~\ref{A} is equivalent to~\ref{AA} is that of hyperbolic groups.
Beyond Abelian and hyperbolic groups, Krat proved the equivalence  for   word metrics on  the discrete Heisenberg group  $\HH (\Z)$,
\cite{MR1758422}.
For general subFinsler metrics on Carnot groups  it has been proven in \cite{MR3153949}, following \cite{MR1666070}, that~\ref{A} is equivalent to the fact that the projections onto $\HH/[\HH,\HH]$ of the corresponding unit balls coincide, see~\ref{TEO103} below.
Our first result shows that this last condition is equivalent to each one of~\ref{A} and~\ref{AA} in the case of the Heisenberg group endowed with subRiemannian metrics.

\begin{Teo}\label{TEO1}
Let $d$ and $d'$ be two left-invariant subRiemannian distances on $\HH$ whose horizontal spaces are $(V,g)$ and $(V',g')$ respectively. 
Let $\pi:\hh\to\hh/[\hh,\hh]$ be the quotient projection and $\hat\pi:\HH\to\hhh$ the corresponding group morphism.

Then the following assertions are equivalent:
\begin{enumerate}[label=(\alph*)]
\item\label{TEO101} 	there exists $c>0$ such that $|d-d'|< c$;
\vspace{1mm}

\item\label{TEO102}	$\frac{d(p,q)}{d'(p,q)} \to 1$ when $d(p,q)\to \infty$;
\vspace{1mm}

\item\label{TEO103} 	$\hat\pi\left(\{p\in\HH:\ d(0,p)\le R\}\right) = \hat\pi\left(\{p\in\HH:\ d'(0,p)\le R\}\right)$, for all $R>0$, here $0$ denotes the neutral element of $\HH$;
\vspace{1mm}

\item\label{TEO104} 	$\pi\left(\{v\in V:\ g(v,v)\le 1\}\right) = \pi\left(\{v'\in V':\ g'(v',v')\le 1\}\right)$;
\vspace{1mm}

\item\label{TEO105} 	there exists a scalar product $ \bar g $ on $\hh/[\hh,\hh]$ such that both
\vspace{-2mm}
		\[
		\pi|_{V} : (V, g) \to (\hh/[\hh,\hh], \bar g) 
		\quad\text{ and }\quad
		\pi|_{V'} : (V', g') \to (\hh/[\hh,\hh], \bar g) 
		\]

\vspace{-2mm}
are submetries.
\end{enumerate}
\end{Teo}
%

Next, we prove that in every class of the equivalence relation~\ref{AAA} there is exactly one strictly subRiemannian metric.
To every left-invariant subRiemannian metric $d$ we define the \emph{associated subRiemannian metric} $d'$ as follows.
If $d$ is Riemannian
 defined by a scalar product $g$ on $\hh$,
 then $d'$ is
 the strictly subRiemannian metric for which 
the  horizontal space
$V$ is $g$-orthogonal to $[\hh,\hh]$ and the scalar product is $g|_V$. 
If $d$ is strictly subRiemannian, then $d'=d$.

\begin{Teo}\label{TEO2}
 	Let $d$ and $d'$ be two left-invariant subRiemannian distances on $\HH$.
	Their associated subRiemannian metrics are the same if and only if
	\begin{equation}\label{eq03281831}
	\lim_{d(p,q)\to\infty}|d(p,q)-d'(p,q)|=0  .
	\end{equation}
%
	Moreover, if \eqref{eq03281831} holds, then there is $C>0$ such that 
	\begin{equation}\label{eq03251459}
	|d(p,q) - d'(p,q)| \le \frac{C}{d(p,q)} , 
	\qquad\forall p,q\in\HH .
	\end{equation}
\end{Teo}


We remark that the estimate \eqref{eq03251459} in Theorem~\ref{TEO2} is sharp,  as we will show  in Remark~\ref{sharp}. 

The above result can be interpreted in terms of asymptotic cones.
Namely, if $d$ is a left-invariant Riemannian metric on $\HH$ and $d'$ is the associated subRiemannian metric, 
then $(\HH,d')$ is the asymptotic cone of $(\HH,d)$.
For the analogous result in arbitrary nilpotent groups see \cite{MR741395}. 
By Theorem~\ref{TEO2}, more is true: $(\HH,d)$ is at bounded distance from $(\HH,d')$.
Notice that this consequence cannot be deduced by 
the similar results for discrete subgroups of the Heisenberg group in
 \cite{MR1758422} and \cite{2014arXiv1411.4201D},
 because
  the word metric  is only quasi-isometric to the Riemannian one.
Moreover, we remark that there are examples 
of nilpotent groups of step two that are not at bounded distance from their asymptotic cone,
see \cite{MR3153949}. 

We now focus on the horoboundary.
As a consequence of Theorem~\ref{TEO2} and of the results of Klein-Nikas   \cite{MR2738530}, we get : 

\begin{Cor}\label{TEO3}
	If $\dr$ is a left-invariant Riemannian metric on $\HH$ with associated metric $\dcc$, then the horoboundary of $(\HH,\dr)$ coincides with the horoboundary of  $(\HH,\dcc)$;
	hence, it is homeomorphic to a $2$-di\-men\-sio\-nal closed disk  $\bar D^2$.
\end{Cor}

%

More precisely, let $g$ be the scalar product of $\dr$ on $\hh$ and $W\subset\hh$ the orthogonal plane to $[\hh,\hh]$.
Define the norm $\|w\|:=\sqrt{g(w,w)}$ on $W$.
Fix a orthonormal basis $(X,Y)$ for $W$ and set $Z:=[X,Y]\in[\hh,\hh]$, so that $(X,Y,Z)$ is a basis of $\hh$.
We identify $\hh\simeq \HH$ via the exponential map, which is a global diffeomorphism.
So, we write $p=w+zZ$ with $w\in W$ and $z\in\R$ for any point $p\in\HH$.
A diverging sequence of points $\{p_n\}_{n\in\N}\subset\HH$, where $p_n=w_n+z_nZ$, diverges:
\begin{enumerate}
\item 	{\em vertically}, if there exists $M<\infty$ such that $\| w_n \| < M$ for all $n$;

\item  {\em non-vertically with quadratic rate $\nu  \in [-\infty, +\infty]$}, if 
	$w_n$ diverges
	and\footnote{
	From the paper \cite{MR2738530}, there is an extra 4 and a change of sign due to our different choice of coordinates.
	} $\lim_{n \rightarrow \infty}  \frac{z_n}{ 4\|w_n\|^2} = -\nu$.
\end{enumerate}
Then, according to \cite{MR2738530} (see Corollary 5.6, 5.9 and 5.13 therein),  we deduce the   following description of the Riemannian horofunctions:
\begin{description}
\item[(v)]	 a vertically diverging sequence $p_n = w_n+z_nZ$ converges to a  horofunction $h$ if and only if $w_n \rightarrow  w_\infty$, and in this case 
	\[
	h (w+zZ) = \|w_\infty \| - \|w_\infty- w \|;
	\]
	
\item[(nv)] 	a  non-vertically diverging sequence $p_n = w_n+z_nZ$  with quadratic rate $\nu$ converges to a horofunction $h$ if and only if  $\frac{w_n}{\|w_n\|} \rightarrow  \hat w$, and  then 
	\[
	h (w+zZ) =  g (R_{\vartheta }(-\hat w),w)
	\]
where $R_{\vartheta }$ is the anti-clockwise rotation in $W$ of angle $\vartheta=\mu^{-1} (\nu )$, and $\mu: [-\pi, \pi] \rightarrow \overline{\R}$ is the extended Gaveau function 
	\[
	\mu( \vartheta) := \frac{\vartheta - \sin \vartheta \cos \vartheta}{\sin^2(\vartheta)} .
	\]
\end{description}
Moreover,  all the horofunctions of $(\HH,\dr)$ are of type {\bf (v)} or  {\bf (nv)},  by Theorem 5.16 in \cite{MR2738530}; it is also clear that neither is of both types.


In section~\ref{sec24031114} we will also determine the Busemann points of  $\de_h(\HH,\dr)$, that is those horofunctions  obtained by points diverging along quasi-geodesics (see Definition~\ref{def08251540}). We obtain, as in the subRiemannian case:
\begin{Cor}\label{corbusemann}
	The Busemann points of $(\HH,\dr)$ are the horofunctions of type \textbf{(nv)} and can be identified to the boundary of the disk $\bar D^2$.
\end{Cor}

The paper is organized as follows.
In Section~\ref{sec14031114} we introduce the main objects and their basic properties.
In Section~\ref{sec08251559} we estimate the difference between any two strictly subRiemannian left-invariant distances on $\HH$.
In Section~\ref{sec08241035} we compare any Riemannian left-invariant distance on $\HH$ and its associated distance.
At the end of the section we shall prove Theorems~\ref{TEO1} and~\ref{TEO2}.
In Section~\ref{sec24031114} we concentrate on the horofunctions and we prove Corollaries~\ref{TEO3} and~\ref{corbusemann}.
Appendix~\ref{sec23031017} is devoted to the explicit description of subRiemannian geodesics.

\subsubsection*{Acknowledgments}
The initial discussions for this work were done at the
`2013 Workshop on Analytic and Geometric Group Theory` in Ventotene.
We express our gratitude to the organizers:
A.~Iozzi, G.~Kuhn and M.~Sageev.

\section{Preliminaries}\label{sec14031114}

\subsection{Definitions}
The \emph{first Heisenberg group} $\HH$ is the connected, simply connected Lie group associated to the Heisenberg Lie algebra $\hh$.
The \emph{Heisenberg Lie algebra $\hh$} is the only three dimensional nilpotent Lie algebra that is not commutative.
It can be proven that, for any two linearly independent vectors $X,Y\in\hh\setminus[\hh,\hh]$, the triple $(X,Y,[X,Y])$ is a basis of $\hh$ and $[X,[X,Y]]=[Y,[X,Y]]=0$.

We denote by $\mc:T\HH\to\hh$ the left-invariant Maurer-Cartan form.
Namely, denoting by $0$ the neutral element of $\HH$ and identifying $\hh$ with $T_0\HH$, we have $\mc(v) := \dd L_{p}^{-1}v$ for $v\in T_p\HH$, where $L_p$ is the left translation by $p$.

Let $\pi:\hh\to\hhh$ be the quotient projection.
Notice that $\hhh$ is a commutative 2-dimensional Lie algebra.
So the map $\pi$ induces a Lie group epimorphism $\hat\pi:\HH\to\hhh\simeq\HH/[\HH,\HH]$.

\subsection{SubRiemannian metrics in $\HH$}
Let $V\subset\hh$ be a bracket generating subspace.
We have only two cases: either $V=\hh$ or $V$ is a plane and $\hh=V\oplus[\hh,\hh]$.
In both cases the restriction of the projection $\pi|_V:V\to\hhh$ is surjective.

Let $g$ be a scalar product on $V$ and set the corresponding norm $\|v\|:=\sqrt{g(v,v)}$ for $v\in V$.

An absolutely continuous curve $\gamma:[0,1]\to\HH$ is said \emph{horizontal} if $\mc(\gamma'(t))\in V$ for almost every $t$.
For a horizontal curve we have the \emph{length}
\[
\ell(\gamma) := \int_0^1\|\mc(\gamma'(t))\|\dd t .
\]
A \emph{subRiemannian metric} $d$ is hence defined as
\[
d(p,q) := \inf\left\{\ell(\gamma):\gamma\text{ horizontal curve from $p$ to $q$}\right\} .
\]

SubRiemannian distances on $\HH$ are complete, geodesic, and left-in\-va\-riant.
They are either Riemannian, when $V=\hh$, or \emph{strictly subRiemannian}, when $\dim V=2$.
The pair $(V,g)$ is called the \emph{horizontal space of} $d$.

Since $\pi|_V:V\to \hhh$ is surjective, it induces a norm $\|\cdot\|$ on $\hhh$ such that $\pi:(V,\|\cdot\|)\to(\hhh,\|\cdot\|)$ is an submetry, i.e., for all $w\in\hhh$ it holds $\|\pi(w)\| = \inf\{\|v\|:\pi(v)=w\}$.
Here we use the same notation for norms on $V$ and on $\hhh$, because there will be no possibility of confusion.
The norm on $\hhh$ is characterized by
\begin{equation}\label{eq05081258}
\pi\left(\{v\in V:\|v\|\le R\}\right) 
= \{w\in\hhh:\|w\|\le R\},
\end{equation}
for all $R>0$.
\begin{Prop}\label{prop23030248}
 	Let $d$ be subRiemannian metric on $\HH$ with horizontal space $(V,g)$.
	Then for all $R>0$
	\[
	\pi\left(\{v\in V:\|v\|\le R\}\right) 
	= \hat\pi\left(\{p\in\HH:d(0,p)\le R\}\right) .
	\]
	In particular, $\hat\pi:(\HH,d)\to(\hhh,\|\cdot-\cdot\|)$ is a submetry, i.e., for all $v,w\in\hhh$
	\[
	\|v-w\| = \inf\{d(p,q):\hat\pi(p)=v,\ \hat\pi(q)=w\} .
	\]
\end{Prop}
\begin{proof}
 	\ci
	Let $v\in V$ with $\|v\|\le R$.
	Set $\gamma(t):=\exp(t v)$.
	Then $\gamma:[0,1]\to\HH$ is a horizontal curve with $d(0,\exp(v))\le\ell(\gamma) =\|v\| \le R$.
	Since $\hat\pi(\exp(v)) =\pi(v)$, then we have proven this inclusion.
	
	\ic
	Let $p\in\HH$ with $d(0,p)\le R$ and let $\gamma:[0,T]\to\HH$ be a $d$-length-minimizing curve from $0$ to $p$ parametrized by arc-length, so $T = d(0,p)$.
	Then $\hat\pi\circ\gamma:[0,T]\to\hhh$ is a curve from $0$ to $\pi(p)$ and
	\begin{align*}
		\|\hat\pi(p)\| 
		&\le \int_0^T\|(\hat\pi\circ\gamma)'(t)\|\dd t \\
		&= \int_0^T \|\pi\circ\mc(\gamma'(t))\| \dd t \\
		&\le \int_0^T \|\mc(\gamma'(t))\| \dd t \\
		&= \ell(\gamma)
		= d(0,p).
	\end{align*}
	In the first equality we used the fact that $\hat\pi$ is a morphism of Lie groups and its differential is $\pi$, i.e., $\omega_{\HH/[\HH,\HH]}\circ\dd\hat\pi = \pi\circ\mc$, where $\omega_{\HH/[\HH,\HH]}$ is the Mauer-Cartan form of $\HH/[\HH,\HH]$.
\end{proof}

\begin{Prop}\label{prop23030249}
 	Let $d,d'$ be two subRiemannian metrics on $\HH$ such that 
	\[
	\lim_{p\to\infty} \frac{d(0,p)}{d'(0,p)} = 1 .
	\]
	Then
	\begin{equation}\label{eq05081521}
	\hat\pi\left(\{p\in\HH:d(0,p)\le R\}\right)
	= \hat\pi\left(\{p\in\HH:d'(0,p)\le R\}\right) .
	\end{equation}
\end{Prop}
\begin{proof}
 	Let $\|\cdot\|$ and $\|\cdot\|'$ be the norms on $\hhh$ induced by $d$ and $d'$, respectively.
	We will show that
	\begin{equation}\label{eq23030240}
	 	\lim_{v\to\infty} \frac{\|v\|}{\|v\|'} = 1,
	\end{equation}
	which easily implies $\|\cdot\|=\|\cdot\|'$ and \eqref{eq05081521} by \eqref{eq05081258} and Proposition~\ref{prop23030248}.
	
	Since both maps $\hat\pi:(\HH,d)\to(\hhh,\|\cdot\|)$ and $\hat\pi:(\HH,d')\to(\hhh,\|\cdot\|')$ are submetries, for every $v\in\hh/[\hh,\hh]$ there are $p_v,p_v'\in\HH$ such that $\hat\pi(p_v)=\hat\pi(p'_v) = v$, $\|v\|=d(0,p_v)$ and $\|v\|'=d'(0,p_v')$. 
	
	Moreover it holds $\|v\|' \le d'(0,p_v)$ and $\|v\|\le d(0,p_v')$, again because $\hat\pi$ is a submetry in both cases.
	Therefore
	\[
	 \frac{d(0,p_v)}{d'(0,p_v)}
	 \le \frac{\|v\|}{\|v\|'} \le 
	 \frac{d(0,p_v')}{d'(0,p_v')}
	\]
	Finally, if $v\to\infty$, then both $d(0,p_v)$ and $d(0,p_v')$ go to infinity as well. 
	The relation \eqref{eq23030240} is thus proven.
\end{proof}

\subsection{Balayage area and lifting of curves}
Let $V\subset\hh$ be a two-dimensional subspace with $V\cap[\hh,\hh]=\{0\}$.
Then $[\hh,\hh]=[V,V]$, i.e., $V$ is bracket generating.
Moreover, $\pi|_V:V\to\hhh$ is an isomorphism.

If $\rho:[0,T]\to\hhh$ is a curve with $\rho(0)=0$, then there is a unique $\tilde\rho:[0,T]\to\HH$ such that
\[
\begin{cases}
 	\tilde\rho(0)=0,\\
	\mc(\tilde\rho'(t)) = \pi|_V^{-1}(\rho(t)') .
\end{cases}
\]
Since $(\pi\circ\tilde\rho)'=\rho'$, then $\pi\circ\tilde\rho=\rho$.
So, $\tilde\rho$ is called the \emph{lift} of $\rho$.

The previous ODE system that defines $\tilde\rho$ can be easily integrated.
Let $X,Y\in V$ be a basis, set $Z:=[X,Y]$, so that $(X,Y,Z)$ is a basis of $\hh$.
Let $(x,y,z)=\exp(xX+yY+zZ)$ be the exponential coordinates on $\HH$ defined by $(X,Y,Z)$.
Using the Backer-Campbell-Hausdorff formula, one shows that $X,Y,Z$ induce the following left-invariant vector fields on $\HH$:
\[
\hat X = \de_x - \frac y2 \de_z,
\qquad
\hat Y = \de_y + \frac x2 \de_z,
\qquad
\hat Z = \de_z .
\]
Thanks to these vector fields, we can describe the Maurer-Cartan form as
\[
\mc(a\hat X + b\hat Y + c\hat Z) = aX+bY+cZ .
\]
The lift of $\rho$ is hence given by the ODE
\[
\begin{cases}
 	\tilde\rho_1' = \rho_1', \\
	\tilde\rho_2' = \rho_2', \\
	\tilde\rho_3' = \frac12 \left(\rho_1\rho_2' - \rho_2\rho_1' \right) .
\end{cases}
\]
Take the coordinates $(x,y)$ on $\hhh$ given by the basis $(\pi(X),\pi(Y))$ and define the \emph{balayage area} of a curve $\rho:[0,T]\to\hhh$ as
\begin{equation}
 	\BA(\rho)= \frac12 \int_\rho (x\dd y - y\dd x) .
\end{equation}
If $\rho(0)=0$, then the balayage area of $\rho$ corresponds to the signed area enclosed between the curve $\rho$ and the line passing through $0$ and $\rho(T)$.

It follows that 
\[
\tilde\rho(t) = \left( \rho_1(t) , \rho_2(t) , \BA(\rho|_0^t) \right) .
\]
In an implicit form we can write
\begin{equation}
\tilde\rho(t) = \exp \left( (\pi|_V)^{-1}(\rho(t)) + \BA(\rho|_0^t) Z \right) .
\end{equation}
Notice that the lift $\tilde\rho$ of a curve $\rho$ depends on the choice of $V$.
Moreover, both the area and the Balayage area in $\hhh$ depend on the choice of the basis $(X,Y)$.
Nevertheless, once a plane $V\subset\hh$ is fixed, the lift $\tilde\rho$ does not depend on the choice of the basis $X,Y$.

If $g$ is a scalar product on $V$ and $d$ is the corresponding strictly subRiemannian metric, the balayage area gives a characterization of $d$-length-minimizing curves. 
Let $\bar g$ be the scalar product on $\hhh$ induced by $g$.
Then the $d$-length of a curve $\tilde\rho:[0,T]\to\HH$ equals the length of $\rho=\pi\circ\tilde\rho$.

Therefore, given $p=(x,y,z)\in\HH$, we have
\[
d(0,p) = \inf\left\{ \ell(\rho):\rho:[0,1]\to\hhh,\rho(0)=0,\rho(1)=\hat\pi(p),\BA(\rho)=z \right\} .
\]
This express the so-called \emph{Dido's problem} in the plane, and the solutions are arc of circles.
It degenerates into a line if $z=0$.
We can summarize the last discussion in the following result.
\begin{Lem}
 	A curve $\tilde\rho:[0,1]\to\HH$ is $d$-length-minimizing from $0$ to $p=(x,y,z)$ if and only if $\rho:=\hat\pi\circ\tilde\rho$ is an arc of a circle from $0$ to $\hat\pi(p)$ with $\BA(\rho)=z$.
\end{Lem}

\section{Comparison between strictly subRiemannian metrics}\label{sec08251559}

The present section is devoted to comparing strictly subRiemannian distances.
For such distances,
Proposition~\ref{prop22031857} gives the only non-trivial implication in Theorem~\ref{TEO1}.
The general case will follow from Proposition~\ref{prop23031006}.



\begin{Prop}\label{prop22031857}
 	Let $d$ and $d'$ be two strictly subRiemannian metrics on $\HH$ with horizontal spaces $(V,g)$ and $(V',g')$, respectively.
	Suppose that there exists a scalar product $ \bar g $ on $\hh/[\hh,\hh]$ such that both 
		\[
		\pi|_{V} : (V, g) \to (\hh/[\hh,\hh], \bar g) 
		\quad\text{ and }\quad
		\pi|_{V'} : (V', g') \to (\hh/[\hh,\hh], \bar g) 
		\]
	are submetries.
	
	Then
	\begin{equation}\label{eq22031438}
	\sup_{p\in\HH}|d(0,p)-d'(0,p)| < \infty.
	\end{equation}
	
	Moreover, if $d\neq d'$, then 
	\begin{equation}\label{eq22031439}
	\limsup_{p\to\infty} |d(0,p)-d'(0,p)| > 0 .
	\end{equation}
\end{Prop}
In the proof we will give the exact value of the supremum in \eqref{eq22031438}.
Indeed, by \eqref{eq23031109} and \eqref{eq22031519}, we get $\sup_{p\in\HH}|d(0,p)-d'(0,p)|=2|h|$, where $h$ is defined below.

For \eqref{eq22031438} we will first prove that two of such subRiemannian distances are one the conjugate of the other and then we apply Lemma~\ref{lem22031443}.

For \eqref{eq22031439} we will give a sequence $p_n\to\infty$ and a constant $c>0$ such that $|d(0,p_n)-d'(0,p_n)|>c$ for all $n\in\N$.\\

\subsection{Proof of \eqref{eq22031438}}
Since $\dim V=\dim V'=2$, then $\pi|_V$ and $\pi|_{V'}$ are isomorphisms. 
Therefore by the assumption they are isometries onto $(\hhh,\bar g)$.

Let $X\in V\cap V'$ be with $g(X,X)=1$.
Then $g'(X,X)=1$ as well.

Let $Y\in V$ be orthogonal to $X$ with $g(Y,Y)=1$.
Then $Z:=[X,Y]\neq0$ and $(X,Y,Z)$ is a basis of $\hh$.

Let $Y':=\pi|_{V'}^{-1}(\pi(Y))\in V'$.
Then $g'(Y',Y')=1$ and $g'(X,Y')=0$.
Moreover, there is $h\in\R$ such that $Y'=Y+hZ$.
In particular, $[X,Y']=Z$.

Using the formula $\rm{Ad}_{\exp(hX)}(v) = e^{\rm{ad}_{hX}}v = v + h[X,v]$, we notice that 
\begin{equation}
\begin{cases}
	\rm{Ad}_{\exp(hX)}(X) = X, \\
 	\rm{Ad}_{\exp(hX)}(Y)=Y', \\
	\rm{Ad}_{\exp(hX)}(Z)=Z .
\end{cases}
\end{equation}
In particular $\rm{Ad}_{\exp(hX)}|_V:(V,g)\to(V',g')$ is an isometry.

Therefore, the conjugation
\[
C_{\exp(hX)}(p) := \exp(hX)\cdot p\cdot\exp(hX)^{-1}
\]
is an isometry $C_{\exp(hX)}:(\HH,d)\to(\HH,d')$.

We can now use the following Lemma~\ref{lem22031443} and get
\begin{equation}\label{eq23031109}
\sup_{p\in\HH}|d(0,p)-d'(0,p)| \le 2|h| .
\end{equation}
\begin{Lem}\label{lem22031443}
 	Let $G$ be a group with neutral element $e$ and let $d,d'$ be two left-invariant distances on $G$.
	
	If there is $g\in G$ such that for all $p\in G$
	\[
	d'(e,p) = d(e,gpg^{-1}) ,
	\]
	then for all $p\in G$
	\[
	|d(e,p) - d'(e,p)| \le 2 \max\{d(e,g),d'(e,g)\}.
	\] 
\end{Lem}
\begin{proof}
 	Note that, since $d$ is left invariant, then for all $a,b\in G$ we have $d(e,ab) \le d(e,a) + d(e,b)$ and $d(e,a) = d(e,a^{-1})$.
	
	Hence $d(e,p) 
		= d(e, g^{-1}gpg^{-1}g) 
		\le d(e,g^{-1}) + d(e,gpg^{-1}) + d(e,g) 
		= 2 d(e,g) + d'(e,p)$.
	The other inequality follows by symmetry.
\end{proof}

\subsection{Proof of \eqref{eq22031439}}
We keep the same notation of the previous subsection.
Up to switching $V$ with $V'$, we can assume $h>0$.

Let $(x,y,z)$ be the exponential coordinates on $\HH$ induced by the basis $(X,Y,Z)$ of $\hh$, i.e., $(x,y,z) = \exp(xX+yY+zZ) \in\HH$.
Similary, on $\hhh$ we have coordinates $(x,y)=x\pi(X)+y\pi(Y)$.

For $R>0$, define
\[
p_R := \left(0 , 2R , \frac{\pi R^2}{2} + 2hR \right) .
\]
We will show that
\begin{equation}\label{eq22031519}
\lim_{R\to\infty} d(0,p_R) - d'(0,p_R) = 2h .
\end{equation}

\begin{figure}[b]
\includegraphics[width=\textwidth]{PictureOfCC-CC03.pdf}
\caption{Curves in $\hhh$ for the proof of \eqref{eq22031439}.}
\end{figure}

Fix $R>0$.
Let $\gamma:[0,T]\to\HH$ be a $d'$-minimizing curve from $0$ to $p_R$.
Then $\hat\pi\circ\gamma:[0,T]\to\hhh$ 
is half circle of center $(0,R)$ and radius $R$.
The balayage area of $\hat\pi\circ\gamma$ is
\[
\BA(\hat\pi\circ\gamma) 
= \frac{\pi R^2}{2} .
\]

Let $\eta:[0,T]\to\HH$ be the $d$-length-minimizing curve from $0$ to $p_R$.
Then $\hat\pi\circ\eta:[0,T]\to\hhh$ is an arc of a circle of radius $S_R$ whose balayage area is
\begin{equation}\label{CE1156}
 	\BA(\hat\pi\circ\eta) 
= \frac{\pi R^2}{2} + 2Rh 
= \BA(\hat\pi\circ\gamma) + 2Rh .
\end{equation}

It is clear that $S_R>R$ and that
the circle of $\hat\pi\circ\eta$ has center $(\mu_R,R)$ for some $\mu_R>0$.
So we have
\[
S_R^2 = R^2 +\mu_R^2 .
\]
It is also clear from the picture that
\begin{equation}\label{CE1157}
 	\frac{\pi S_R^2}{2} + 2R\mu_R 
	\le \BA(\hat\pi\circ\eta)
	\le \frac{\pi S_R^2}{2} + 2S_R\mu_R .
\end{equation}

Now, let's look at the lengths.
First of all, notice that $\ell_{d'}(\gamma)=\ell(\hat\pi\circ\gamma)$ and $\ell_d(\eta)=\ell(\hat\pi\circ\eta)$.
For one curve we have
\[
\ell(\hat\pi\circ\gamma) = \pi R,
\]
for the other we have the estimate
\[
\pi S_R + 2\mu_R 
\le \ell(\hat\pi\circ\eta),
\]
which is clear from the picture.
Hence
\begin{align*}
	\liminf_{R\to\infty} d(0,p) - d'(0,p)
	&= \liminf_{R\to\infty} \ell(\hat\pi\circ\eta)-\ell(\hat\pi\circ\gamma) \\
	&\ge \lim_{R\to\infty}  \pi S_R + 2\mu_R - \pi R  \\
	&= \lim_{R\to\infty} \pi (S_R - R) + 2\mu_R .
\end{align*}
We claim that
\begin{equation}\label{eq22031826}
 	\lim_{R\to\infty} \pi (S_R - R) + 2\mu_R = 2h .
\end{equation}
Let us start by checking that,
\begin{equation}\label{eq22031828}
 	\mu_R < h .
\end{equation}
Indeed, from the first inequality of \eqref{CE1157} together with \eqref{CE1156} it follows
\[
\frac{\pi S_R^2}{2} + 2R\mu_R 
\le 
\frac{\pi R^2}{2} + 2hR .
\]
Since $S_R>R$, then
\[
0 < \frac{\pi S_R^2}{2} - \frac{\pi R^2}{2} \le 2R (h-\mu_R) ,
\]
i.e., the
inequality \eqref{eq22031828}.

From the second inequality of \eqref{CE1157} together with \eqref{CE1156} we get
\[
\frac{\pi R^2}{2} + 2Rh
\le
\frac{\pi S_R^2}{2} + 2S_R\mu_R .
\]
Using the facts $\mu_R\le h$ and $S_R\le R+\mu_R\le R+ h$, from this last inequality one gets
\begin{align}
0\le
2 (h - \mu_R)
&\le (S_R-R) \frac1R \left(\frac\pi2(S_R+R) + 2\mu_R\right) \nonumber \\
&\le
(S_R-R) \left( \pi + \frac hR ( \frac\pi2 + 2) \right) \label{eq22031851}
\end{align}

Moreover, since $h^2\ge \mu_R^2 = S_R^2 - R^2 = (S_R-R) (S_R+R)$, we also have 
\begin{equation}\label{eq22031854}
\lim_{R\to\infty}(S_R-R) = 0 .
\end{equation}

Finally, from \eqref{eq22031854} and \eqref{eq22031851} we obtain \eqref{eq22031826}, as claimed.
This completes the proof of \eqref{eq22031519} and of Proposition~\ref{prop22031857}.

\section{Comparison between Riemannian and strictly subRiemannian metrics}\label{sec08241035}

Let $\dr$ be a Riemannian metric on $\HH$ with horizontal space $(\hh,g)$.

Let $V\subset\hh$ be the plane orthogonal to $[\hh,\hh]$ and let $\dcc$ be the strictly subRiemannian metric on $\HH$ with horizontal space $(V,g|_V)$.

Fix a basis $(X,Y,Z)$ for $\hh$ such that $(X,Y)$ is an orthonormal basis of $(V,g|_V)$ and $Z=[X,Y]$.
The matrix representation of $g$ with respect to $(X,Y,Z)$ is
\[
g=
\begin{pmatrix}
 	1 & 0 & 0 \\
	0 & 1 & 0 \\
	0 & 0 & \zeta^2
\end{pmatrix}
\]
where $\zeta>0$.

Let $\dcc$ be the strictly subRiemannian metric on $\HH$ with horizontal space $(V,g|_V)$.

Our aim in this section is to prove the following proposition.
\begin{Prop}\label{prop23031006}
 	If $\dcc(0,p)$ is large enough,
	then:
	\begin{equation}\label{eq23031008}
	0\le \dcc(0,p) - \dr(0,p) \le \frac{4\pi^2}{\zeta^2} \frac1{\dcc(0,p) - \frac{2^{3/2}\pi}{\zeta}} .
	\end{equation}
	In particular it holds
	\begin{equation}
	\lim_{p\to\infty} \left|\dcc(0,p) - \dr(0,p)\right| = 0 .
	\end{equation}
\end{Prop}

For the proof of this statement, we need to know length-minimizing curves for $\dr$ and $\dcc$, and a few properties of those, see the exposition in the Appendix~\ref{sec23031017}.

\begin{proof}
Let $(x,y,z)$ be the exponential coordinates on $\HH$ induced by the basis $(X,Y,Z)$ of $\hh$, i.e., $(x,y,z) = \exp(xX+yY+zZ) \in\HH$.
Fix $p = (p_1,p_2,p_3)\in\HH$.

Notice that both $\dr$ and $\dcc$ are generated as length metrics using the same length measure $\ell$, with the difference that $\dr$ minimizes the length among all the curves, while $\dcc$ takes into account only the curves tangent to $V$.
This implies that
\[
\forall p,q\in\HH
\qquad
\dcc(p,q) \ge \dr(p,q) ,
\]
therefore we get the first inequality in \eqref{eq23031008}.
We need to prove the second inequality of \eqref{eq23031008}.

If $p\in\{z=0\}$, then $\dcc(0,p)=\dr(0,p)$ by Corollary~\ref{cor1103}, and the thesis is true.

Suppose $p\notin\{z=0\}$ and let $\gamma:[0,T]\to\HH$ be a $\dr$-length minimizing curve from $0=\gamma(0)$ to $p=\gamma(T)$.
Since $p\notin\{z=0\}$ and since we supposed $p$ to be far away enough, then by Corollary~\ref{corGeoDr1bis} we can parametrize $\gamma$ in such a way that $\gamma$ is exactly in the form expressed in \textsc{Type II} in Proposition~\ref{teoGeoDr} for some $k>0$ and $\theta\in\R$.

By Corollary~\ref{corGeoDr1} it holds
\begin{equation}\label{eq24031017}
kT \le 2\pi.
\end{equation}
Moreover, by Corollary~\ref{corGeoDr}
\begin{equation}\label{eq1054}
 	\dr(0,p) = \|\mc(\gamma')\|\cdot T =  \sqrt{1+\frac{ k^2}{\zeta^2} } \cdot T.
\end{equation}

Let $\eta:[0,T]\to\HH$ be the $\dcc$-length-minimizing curve corresponding to $\gamma$ as shown in Corollary~\ref{corGeoDr}.
Then we know that $\dcc(0,\eta(T)) = \ell(\eta)=T$, and
\begin{equation}\label{eq1530}
 	p = \gamma(T) = \eta(T) + (0,0,\frac{kT}{\zeta^2}).
\end{equation} 
Hence by Corollary~\ref{corGeoDr} and \eqref{eq24031017}
\[
\dcc(0,p) \le \dcc(0,\eta(T)) + \dcc(\eta(T),\gamma(T)) \le T + \frac{2^{3/2}\pi}{\zeta} ,
\]
i.e.,
\begin{equation}\label{eq1657}
\frac1T \le \frac{1}{\dcc(0,p) - \frac{2^{3/2}\pi}{\zeta}}.
\end{equation}

Since $\eta$ is a $\dcc$-rectifiable curve, then
$
\eta(T)_3 = \BA(\hat\pi\circ\eta) 
$, 
where $\eta(T)_3$ is the third coordinate of the point in the exponential coordinates.
Since $\hat\pi\circ\gamma = \hat\pi\circ\eta$, then we have by \eqref{eq1530}
\begin{equation}\label{eq1552}
p_3 
= \BA(\hat\pi\circ\gamma) + \frac{kT}{\zeta^2}.
\end{equation}

Notice that $\hat\pi\circ\gamma$ is an arc of a circle in $\hhh$ of radius $\frac1k$, see Proposition~\ref{teoGeoDr}.

Now we want to define a horizontal curve $\tilde\rho:[-\epsilon,T+\epsilon]\to\HH$, where $\epsilon>0$ has to be chosen, such that $\tilde\rho(-\epsilon)=0$ and $\tilde\rho(T+\epsilon)=p$.
We first define a curve $\rho:[-\epsilon,T+\epsilon]\to\hhh$ and then take its lift to $\HH$.

\begin{figure}
 	\includegraphics[width=0.49\textwidth]{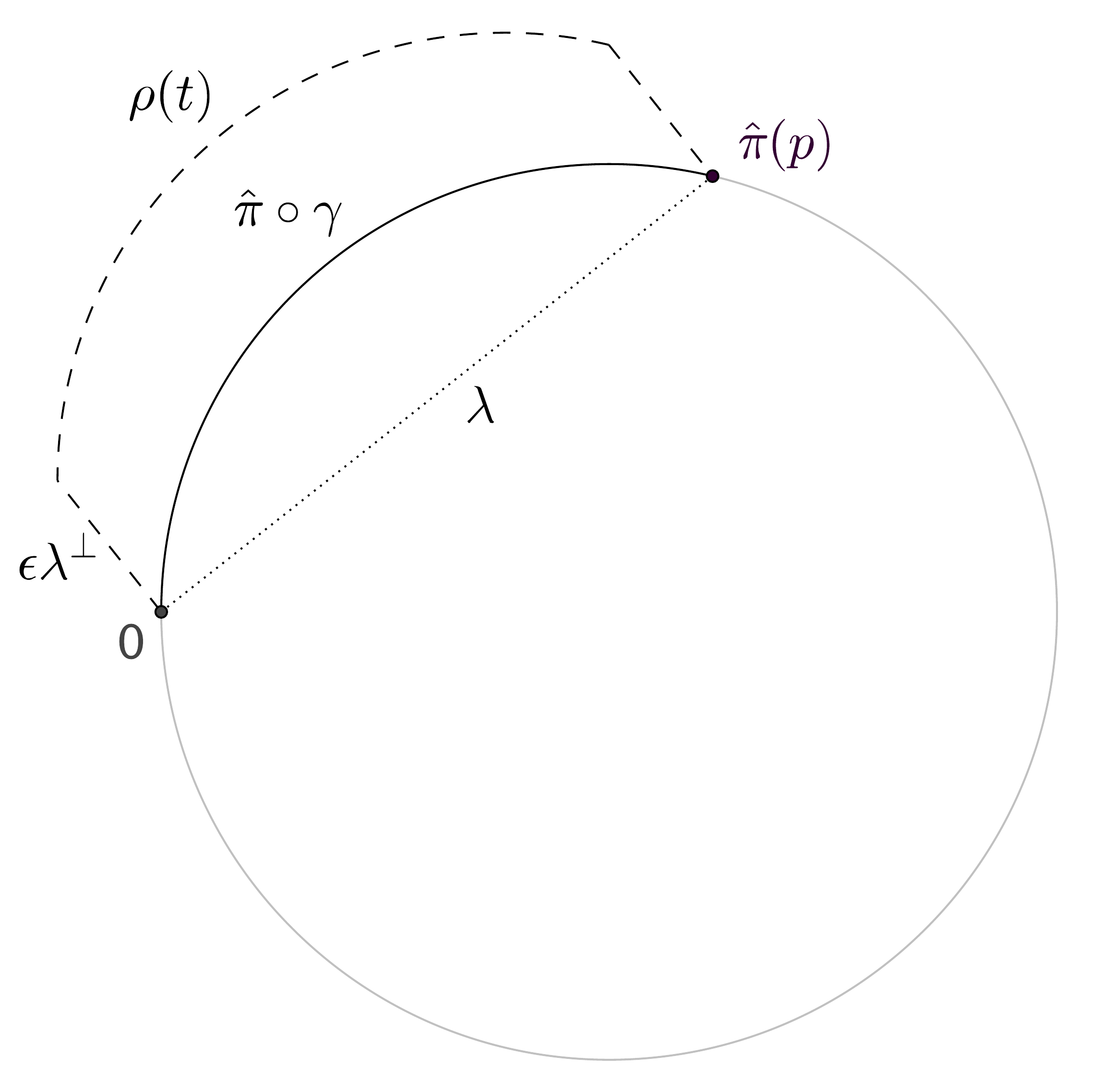}
	\includegraphics[width=0.5\textwidth]{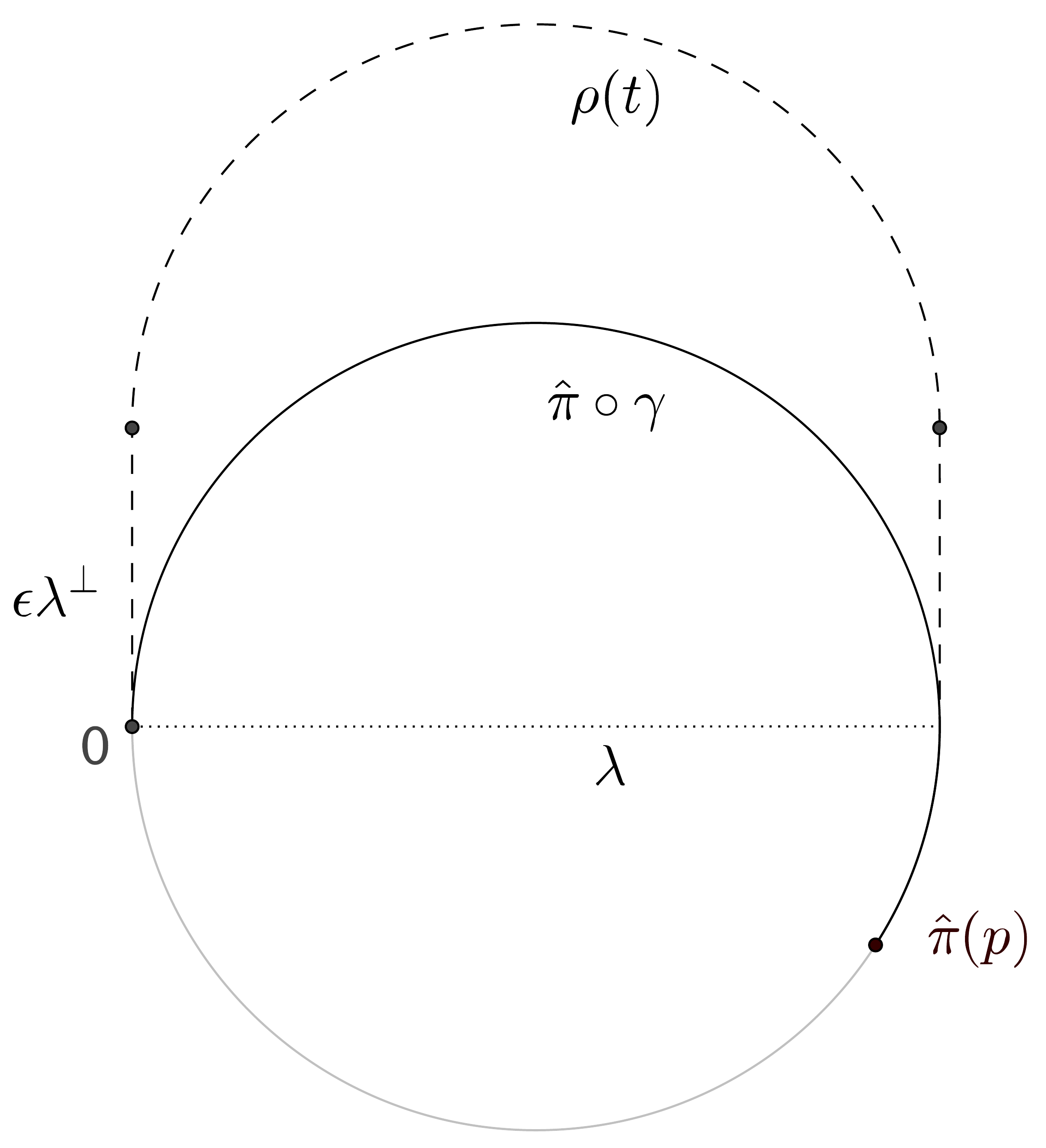}
	\caption{Curves for Case 1 and Case 2.}
\end{figure}

For the definition of $\rho$ we follow two different strategies for two different cases:\\

\textsc{Case 1.}
Suppose that $\hat\pi\circ\gamma$ doesn't cover the half of the circle, i.e., $T\le\frac\pi k$. 
Set $\lambda = \hat\pi(p)\in\hhh$.
Then $T$ is smaller than the circle of diameter $\|\lambda\|$, i.e., 
\begin{equation}\label{eq1110}
\|\lambda\| \ge \frac T\pi .
\end{equation}
Let $\lambda^\perp\in\hhh$ be the unitary vector perpendicular to $\lambda$
and forming an angle smaller than $\pi/2$ with the arc $\hat\pi\circ\gamma$.
Let $\epsilon>0$ such that
\begin{equation}\label{eq1112}
\epsilon \cdot\|\lambda\| = \frac{kT}{\zeta^2} .
\end{equation}
Now, define $\rho:[-\epsilon,T+\epsilon]\to\hhh$ as
\[
\rho(t) = 
\begin{cases}
 	(t+\epsilon) \lambda^\perp							&\text{ for }-\epsilon\le t \le 0 \\
	\epsilon\lambda^\perp + \hat\pi\circ\gamma(t)					&\text{ for }0\le t\le T \\
	\epsilon\lambda^\perp + \hat\pi\circ\gamma(T) - (t-T) \lambda^\perp	&\text{ for }T\le t\le T+\epsilon
\end{cases}
\]
Notice that
\[
\BA(\rho) = \BA(\hat\pi\circ\gamma) + \epsilon\cdot\|\lambda\| = \BA(\hat\pi\circ\gamma) + \frac{kT}{\zeta^2} \overset{\eqref{eq1552}}= p_3
\]
and that $\rho(T+\epsilon) = \hat\pi\circ\gamma(T) = \hat\pi(p)$.
Then the horizontal lift $\tilde\rho:[-\epsilon,T+\epsilon]\to\HH$ of $\rho$ is a $\dcc$-rectifiable curve from $0$ to $p$.\\

\textsc{Case 2.} 
Suppose that $\hat\pi\circ\gamma$ covers more than half of the circle. 
Let $\lambda\in\hhh$ be the diameter of the circle that containing $0$.
Since $T$ is shorter than the whole circle, then
\begin{equation}\label{eq1556}
\|\lambda\| \ge \frac T\pi .
\end{equation}
Let $\lambda^\perp$ be the unitary vector perpendicular to $\lambda$ 
and forming an angle smaller than $\pi/2$ with the arc $\hat\pi\circ\gamma$.
%
Let $\epsilon>0$ be such that 
\begin{equation}\label{eq1113}
\epsilon\cdot \|\lambda\| = \frac{kT}{\zeta^2} . 
\end{equation}
Now, define $\rho:[-\epsilon,T+\epsilon]\to\hhh$ as
\[
\rho(t) =  
\begin{cases}
 	(t+\epsilon) \lambda^\perp	&\text{ for }-\epsilon\le t \le 0 \\
	\epsilon\lambda^\perp + \hat\pi\circ\gamma(t)	&\text{ for }0\le t\le \frac{\pi\|\lambda\|}{2} \\
	\epsilon\lambda^\perp + \lambda 
		- (t-\frac{\pi\|\lambda\|}{2}) \lambda^\perp	&\text{ for }\frac{\pi\|\lambda\|}{2}\le t\le \frac{\pi\|\lambda\|}{2}+\epsilon	\\
	\hat\pi\circ\gamma(t-\epsilon) &\text{ for }\frac{\pi\|\lambda\|}{2}+\epsilon \le t\le T+\epsilon
\end{cases}
\]
where we used the fact $\lambda=\hat\pi\circ\gamma(\frac{\pi\|\lambda\|}{2}) $.
Notice that
\[
\BA(\rho) = \BA(\hat\pi\circ\gamma) + \epsilon\cdot\|\lambda\| = \BA(\hat\pi\circ\gamma) + \frac{kT}{\zeta^2} = p_3 .
\]
Then the horizontal lift $\tilde\rho:[-\epsilon,T+\epsilon]\to\HH$ of $\rho$ is a $\dcc$-rectifiable curve from $0$ to $p$.\\

In both cases $\tilde\rho$ is a horizontal curve from $0$ to $p$ of length 
\begin{equation}\label{eq1654}
 	\ell(\tilde\rho) = T+2\epsilon
\end{equation}

Moreover, from \eqref{eq1110} and \eqref{eq1112} (respectively \eqref{eq1556} and \eqref{eq1113}) we get
\begin{equation}\label{eq1656}
	\epsilon 
	= \frac{kT}{\zeta^2\|\lambda\|}  
	\le \frac{kT}{\zeta^2} \frac\pi T 
	\overset{\eqref{eq24031017}}\le \frac{2\pi}{\zeta^2} \frac\pi T 
	= \frac{2\pi^2}{\zeta^2T}
\end{equation}
Finally using in order \eqref{eq1054}, \eqref{eq1654}, \eqref{eq1656}, \eqref{eq1657}
\begin{multline*}
	\dcc(0,p)-\dr(0,p) 
	\le \ell(\tilde\rho) - \sqrt{1+\frac{ k^2}{\zeta^2} } \cdot T \le \\
	\le T + 2\epsilon - \sqrt{1+\frac{ k^2}{\zeta^2} } \cdot T
	\le 2\epsilon
	\le 2 \frac{2\pi^2}{\zeta^2T}
	\le \frac{4\pi^2}{\zeta^2} \frac1{\dcc(0,p) - \frac{2^{3/2}\pi}{\zeta}} .
\end{multline*}
\end{proof}

\begin{Rem}\label{sharp}
 	The inequality \eqref{eq03251459} is sharp.
	Indeed, for $z\to\infty$, we have the asymptotic equivalence
	\begin{equation}\label{eq03301023}
	 	\dcc(0,(0,0,z)) - \dr(0,(0,0,z)) \sim \frac{4\pi^2}{\zeta^2}  \frac1{\dcc(0,(0,0,z)) }.
	\end{equation}
\end{Rem}
\begin{proof}[Proof of \eqref{eq03301023}]
 	We claim that, for $z>0$ large enough,
	\begin{equation}\label{eq03311609}
	\dr(0,(0,0,z)) = 2\sqrt{\pi} \sqrt{  z - \frac\pi{\zeta^2} } .
	\end{equation}
	Let $\gamma:[0,T]\to \HH$ be a $\dr$-length-minimizing curve from $0$ to $(0,0,z)$.
	Since $z$ is large, we assume that $\gamma$ is of (\textsc{Type II}), see Proposition~\ref{teoGeoDr}, for some $k>0$ and $\theta=0$.
	Since the end point is on the $Z$ axis, we have
	\begin{equation}
	kT = 2\pi 
	\end{equation}
	and $z = \frac{T}{2k} + \frac{kT}{\zeta^2}$, from which follows
	\begin{equation}
	T^2  = 4\pi \left( z - \frac{2\pi}{\zeta^2} \right) .
	\end{equation}
	We know also the length of $\gamma$ (see Corollary~\ref{corGeoDr}) and so we get
	\begin{multline*}
	\dr(0,(0,0,z))
	= \ell(\gamma) 
	= T \|\mc(\gamma')\| 
	= T \sqrt{1 + \frac{k^2}{\zeta^2} }
	= \sqrt{ T^2 + \frac{4\pi^2}{\zeta^2} } \\
	= \sqrt{ 4\pi \left( z - \frac{2\pi}{\zeta^2} \right) + \frac{4\pi^2}{\zeta^2} }
	= 2\sqrt{\pi} \sqrt{  z - \frac\pi{\zeta^2}  }.
%
	\end{multline*}
	Claim \eqref{eq03311609} is proved.
	From Corollary~\ref{corGeoDr2} we get $\dcc(0,(0,0,z)) = 2\sqrt{\pi}\sqrt{z}$ and
	\begin{multline*}
	\dcc(0,(0,0,z)) - \dr(0,(0,0,z)) 
	= 2\sqrt{\pi} \left( \sqrt{z} - \sqrt{  z - \frac\pi{\zeta^2}  }
\right) \\
	= \frac{2\sqrt{\pi}}{\sqrt{z}} \frac{ \frac\pi{\zeta^2}  }
	{ 1 + \sqrt{  1 - \frac\pi{\zeta^2z}  } } 
	= \frac{1}{2\sqrt{\pi}\sqrt{z}} 
		\frac{4\pi^2}{\zeta^2}  
		\frac{ 1 }
	{ 1 + \sqrt{  1 - \frac\pi{\zeta^2z}  } } .
	\end{multline*}
\end{proof}

We are now ready to give the proof of the main theorems: 

\begin{proof}[Proof of Theorem~\ref{TEO1}]
	The implication $\ref{TEO101}\THEN\ref{TEO102}$ is trivial.
	The implication $\ref{TEO102}\THEN\ref{TEO103}$ is proven in Proposition~\ref{prop23030249}.
 	The equivalence $\ref{TEO103}\IFF\ref{TEO104}$ follows from Proposition~\ref{prop23030248}.
	The assertion $\ref{TEO105}$ is a restatement of $\ref{TEO104}$.
	For $\ref{TEO104}\THEN\ref{TEO101}$ one uses Proposition~\ref{prop23031006} in order to reduce to the case when both $d$ and $d'$ are strictly subRiemannian and then one applies Proposition~\ref{prop22031857}.
\end{proof}

\begin{proof}[Proof of Theorem~\ref{TEO2}]
 	This is a consequence of Propositions~\ref{prop23031006} and of the sharpness result \eqref{eq22031439} of Proposition~\ref{prop22031857}.
\end{proof}

\section{The horoboundary} \label{sec24031114}

Let $(X,d)$ be a geodesic space and $\Co(X)$ the space of continuous functions $X\to\R$ endowed with the topology of the uniform convergence on compact sets.
The map $\iota:X\into\Co(X)$, $(\iota(x))(y):=d(x,y)$, is an embedding, i.e., a homeomorphism onto its image.

Let $\Co(X)/\R$ be the topological quotient of $\Co(X)$ with kernel the constant functions, i.e., for every $f,g\in\Co(X)$ we set the equivalence relation $f\sim g\IFF f-g$ is constant.

Then the map $\hat\iota:X\into\Co(X)/\R$ is still an embedding.
Indeed, since the map $\Co(X)\to\Co(X)/\R$ is continuous and open, we only need to show that $\hat\iota$ is injective: if $x,x'\in X$ are such that $\iota(x)-\iota(x')$ is constant, then one takes $z\in Z$ such that $d(x,z)=d(x',z)$, which exists because $(X,d)$ is a geodesic space, and checks that 
\[
d(x,x') = \iota(x)(x')-\iota(x')(x')
= \iota(x)(z)-\iota(x')(z)
= 0 .
\]

Define the \emph{horoboundary of $(X,d)$} as
\[
\de_h X := cl(\hat\iota(X)) \setminus \hat\iota(X) \subset\Co(X)/\R ,
\]
where $cl(\hat\iota(X))$ is the topological closure.

Another description of the horoboundary is possible.
Fix $o\in X$ and set
\[
\Co(X)_o := \{f\in\Co(X):f(o)=0\} .
\]
Then the restriction of the quotient projection $\Co(X)_o\to \Co(X)/\R$ is an isomorphism of topological vector spaces.
Indeed, one easily checks that it is both injective and surjective, and that its inverse map is $[f]\mapsto f-f(o)$, where $[f]\in\Co(X)/\R$ is the class of equivalence of $f\in\Co(X)$.

Hence, we can identify $\de_hX$ with a subset of $\Co(X)_o$.
More explicitly: $f\in\Co(X)_o$ belongs to $\de_hX$ if and only if there is a sequence $p_n\in X$ such that $p_n\to\infty$ (i.e., for every compact $K\subset X$ there is $N\in\N$ such that $p_n\notin K$ for all $n>N$) and the sequence of functions $f_n\in\Co(X)_o$,
\begin{equation}\label{eq03251521}
f_n(x) := d(p_n,x)-d(p_n,o),
\end{equation}
converge uniformly on compact sets to $f$.

\begin{proof}[Proof of  Corollary~\ref{TEO3}.]

Let us first remark that if  $d,d'$ are two geodesic distances on $X$ and
\begin{equation}\label{eq23031323}
\lim_{d(p,q)+d'(p,q)\to\infty}|d'(p,q)-d(p,q)| = 0 .
\end{equation}
then 
\[
\de_h(X,d') = \de_h(X,d) .
\]
Indeed, first of all the space $\Co(X)_o$ depends only on the topology of $X$.
Moreover, if $f\in\de_h(X,d)$, let $p_n\in X$ be a sequence as in \eqref{eq03251521} and set $f'_n(x):=d'(p_n,x)-d'(p_n,o)$.
Then
\[
|f'_n(x)-f_n(x)| \le |d'(p_n,x)-d(p_n,x)| + |d'(p_n,o)-d(p_n,o)| ,
\]
and as a consequence of \eqref{eq23031323} we get $f_n'\to f$ uniformly on compact sets.
This shows $\de_h(X,d)\subset\de_h(X,d')$.
The other inclusion follows by the simmetry of \eqref{eq23031323} in $d$ and $d'$.\\
Now, if $\dr$ and $\dcc$ are distances on $\HH$ like in Corollary~\ref{TEO3}, then \eqref{eq23031323} is easily satisfied thanks to Theorem~\ref{TEO2}, and therefore $\de_h(\HH,\dr) = \de_h(\HH,\dr) $ if the Riemannian metric $\dr$ and  the subRiemannian metric $\dcc$ are compatible. The conclusion follows from \cite{MR2738530}.
\end{proof}
  
The Busemann points in the boundary $\de_h(X,d)$  are usually defined as the horofunctions associated to sequences of points $(p_n)$ diverging along rays or ``almost geodesic rays''.  However, in literature  there are   different definitions of almost geodesic rays, according to the generality of the metric space $(X,d)$ under consideration (\cite{MR1375503}, \cite{MR2015055}, \cite{MR2988483}). 
A map $\gamma: I=[0, +\infty) \rightarrow (X,d)$ into a complete  length space   is called

\begin{itemize}[leftmargin=*]
\item 	a {\em quasi-ray}, if the length excess  
	 \[
	 \Delta_N (\gamma) = \sup_{t,s\in  [N,+\infty]} \ell (\gamma; t,s)- d(\gamma(t), \gamma(s))
	 \]
	 tends to zero for $N \rightarrow +\infty$;
\item 	an  {\em almost geodesic ray}, if 
	\[
	\Theta_N (\gamma)= \sup_{t,s\in [N,+\infty]}    d (\gamma(t), \gamma(s)) +  d (\gamma(s), \gamma(0)) -t
	\]
	tends to zero for $N \rightarrow +\infty$.
\end{itemize}
%
%
%
%
(Notice that the second definition depends on the parametrization, while the first one is intrinsic).
We will use here a  notion of Busemann points which is more general than  both of them:

\begin{Def}\label{def08251540}
	A diverging sequence of points $(p_n)$ in  $(X,d)$ is said to diverge {\em almost straightly} if for all $\epsilon>0$ there exists $L$ such that for every $n \geq m \geq L$ we have 
	\begin{equation}\label{eqstraigth}
		d(p_L,p_m) + d(p_m,p_n) -  d(p_L,p_n) < \epsilon
	\end{equation}
	It is easy to verify that points diverging along a quasi-ray or along an almost-geodesic ray diverge almost straightly.
	We then define a {\em Busemann point} as a  horofunction $f$ which is  the limit of a sequence \linebreak $f_n(x)=d(p_n,x)-d(p_n,o)$, for   points $(p_n)$ diverging  almost straightly.
\end{Def} 

To prove Corollary~\ref{corbusemann}, we need the following 

\begin{Lem}\label{lemmabus}
	Let $(X,d)$ be a boundedly compact, geodesic space, $o\in X$ and $\{p_n\}_{n\in\N}\subset X$ a sequence of points diverging almost straightly.
	Then:
	\begin{enumerate}[label=(\roman*)]
	\item	the sequence $f_n (x)=d(p_n,x)-d(p_n,o)$ converges uniformly on compacts to a horofunction $f$;
	\item	$\lim_{n \rightarrow \infty} f(p_n) + d(o, p_n) =0$.
	\end{enumerate}
\end{Lem}
\begin{proof}
	Since the 1-Lipschitz functions $f_n$ are uniformly bounded on compact sets and $(X,d)$ is boundedly compact,
	then the family $\{f_n\}_{n\in\N}$ is pre-compact with respect to the uniform convergence on compact sets.
	Hence, if we prove that there is a unique accumulation point, then we obtain that the whole sequence $\{f_n\}_{n\in\N}$ converges.
	
	So, let $g,g'\in\Co^0(X)$ and let $\{f_{n_k}\}_{k\in\N}$ and $\{f_{n'_k}\}_{k\in\N}$ be two subsequences of $\{f_n\}_{n\in\N}$ such that $f_{n_k}\to g$ and $f_{n_k'}\to g'$ uniformly on compact sets.
	We claim
	\begin{equation}\label{eq08241859}
	\forall \epsilon>0\ \exists R_\epsilon \in \R\ \forall x\in X \quad |g'(x)+R_\epsilon-g(x)| \le \epsilon .
	\end{equation}
	Let $\epsilon>0$.
	Let $L\in\N$ be such that \eqref{eqstraigth} holds. 
	Define
	\begin{align*}
	 	g_L(x) &:= \lim_{k\to\infty} d(p_{n_k},x) - d(p_{n_k},p_L) = g(x)-g(p_L) \\
		g'_L(x) &:= \lim_{k\to\infty} d(p_{n'_k},x) - d(p_{n'_k},p_L) = g'(x)-g'(p_L) .
	\end{align*}
	Then for $n_i\ge n'_j\ge L$, we get
	\begin{multline*}
	 	d(p_{n_i},x) - d(p_{n_i},p_L) - d(p_{n'_j},x) + d(p_{n'_j},p_L) \\
		\le d(p_{n_i}, p_{n'_j}) - d(p_{n_i},p_L) + d(p_{n'_j},p_L) 
		\le \epsilon .
	\end{multline*}
	By taking the limit $i\to\infty$ and $j\to\infty$, we obtain for all $x\in X$
	\[
	g_L(x)-g'_L(x) \le \epsilon .
	\]
	Thanks to the symmetry of the argument, also $g'_L(x)-g_L(x) \le \epsilon$ holds.
	Therefore for all $x\in X$
	\[
	\epsilon 
	\ge |g'_L(x)-g_L(x)|
	= |g'(x) - g(x) + g(p_L) - g'(p_L)|.
	\]
	Setting $R_\epsilon = g(p_L) - g'(p_L)$, we conclude the proof of claim \eqref{eq08241859}.
	
	It is now easy to conclude from \eqref{eq08241859}.
	Indeed, taking $x=o$, we have $|R_\epsilon|\le \epsilon$, therefore for all $\epsilon>0$ and for all $x\in X$
	$
	|g(x)-g'(x)| \le 2\epsilon ,
	$
	i.e., $g=g'$.
	This completes the proof of $(i)$.
	
	To prove assertion (ii), fix $\epsilon>0$ and let $L\in\N$ be as above.
	Then we have for all $n\ge m\ge L$
	\begin{eqnarray*}
	0&\le &d(p_m,p_n) - d(p_n,o) + d(p_m,o) \\
	&=& d(p_m,p_n) + d(p_L, p_m) - d(p_L, p_n) \\
	&& \hspace{3cm}+  d(p_L, p_n) - d(p_n,o) - d(p_L,p_m) + d(p_m,o) \\
	&\le& \epsilon + d(p_L, p_n) - d(p_n,o) - d(p_L,p_m) + d(p_m,o) .
	\end{eqnarray*}
	Taking the limit $n\to\infty$, we obtain the estimate
	\begin{multline*}
	0\le  \liminf_{m\to\infty} f(p_m) + d(p_m,o) \\
	\le \limsup_{m\to\infty} f(p_m) + d(p_m,o) \le \epsilon + f(p_L) - f(p_L) = \epsilon .
	\end{multline*}
	Since $\epsilon>0$ is arbitrary, then $(ii)$ holds true.
\end{proof}

Then, the proof of Corollary~\ref{corbusemann} runs similarly to Theorem 6.5 of \cite{MR2738530}.

\begin{proof}[Proof of  Corollary ~\ref{corbusemann}.]
The horofunctions of type {\bf (nv)} clearly are Busemann points, as they are limit, in particular, of the   Riemannian geodesic rays which are the horizontal half-lines issued from the origin and which are always minimizing, see Proposition~\ref{teoGeoDr} and Corollary~\ref{corGeoDrtypeI} in the Appendix.
On the other hand, consider a horofunction of type {\bf (v)}, $h_u=(v,z) = |u| - |u -v|$, for $u \in W$.
Assume that there exists an almost straightly diverging sequence of points $p_n = v_n + z_nZ$ converging to $h_u$.
By Lemma~\ref{lemmabus} (ii), we deduce that 
\[
\lim_{n\rightarrow \infty} f_u (p_n) + d_R (o,p_n) 
= \lim_{n\rightarrow \infty} |u| - |u - v_n| + d_R (o, p_n) =0 ,
\]
hence $\{v_n\}_{n\in\N}$  is necessarily an unbounded sequence. By Corollary~\ref{TEO3} and the following description of horofunctions, it follows that $h_u$ should be of type {\bf (nv)}, a contradiction.
\end{proof}

\subsubsection{Concluding remarks}

\noindent The  Riemannian Heisenberg group shows a number of counterintuitive features which is worth to stress:

\vspace{1mm}
\noindent (i) 
in view of Corollary~\ref{TEO3}, all Riemannian metrics on $\HH$ with the same associated distance
have the same Busemann functions, though they are not necessarily isometric 
(in contrast, notice that all strictly subRiemannian metrics on  $\HH$ are isometric). 
However, this is not surprising, because all left-invariant Riemannian metrics on $\HH$ are homothetic. 
\vspace{2mm} 

\noindent (ii) 
there exist diverging   sequences of points $\{p_n\}_{n\in\N}$  that {\em visually converge} to a  limit direction $v$ (that is, the minimizing geodesics $\gamma_n$ from $o$ to $p_n$ tend to a limit, minimizing geodesic $\gamma_v$ with initial direction $v$), but whose associated  limit point $h_{\{p_n\}}$ is not given by the limit point  $\gamma_v (+\infty)$ of $\gamma_v$. This happens for all vertically divergent sequences  $\{p_n\}_{n\in\N}$, as the limit geodesic  $\gamma_v$  is horizontal in this case (see Proposition~\ref{teoGeoDr}  and Corollary~\ref{corGeoDr1} in the Appendix).
\vspace{2mm}

\noindent (iii) 
there exist diverging  trajectories   $\{p_n\}_{n\in\N}$, $\{q_n\}_{n\in\N}$ staying at boun\-ded distance from each other, but defining different limit points (e.g., vertically diverging sequences of points with different limit horofunctions).
\vspace{2mm}

\noindent (iv) it is not true that, for a cocompact group of isometries $G$ of $(\HH, \dr)$, the limit set of $G$  (which is the set of accumulation points of an orbit $Gx_0$ in $\de (\HH, \dr)$)  equals the whole Gromov boundary; for instance, the discrete Heisenberg group $G= \HH (\mathbb{Z})$, has a limit set equal to the set of all Busemann points, plus a discrete subset of the interior of the disk boundary $\bar D^2$.
Also, the limit set may depend on the choice of the base point $x_0 \in \HH$.
\vspace{2mm} 

\noindent (v)
The functions appearing in {\bf (nv)} coincide with the  Busemann functions of a Euclidean plane in the direction $R_{\vartheta} (\hat v_\infty)$; that is, the horofunction  $h (v,z)$ associated to a diverging sequence $P_n=(v_n,z_n)$ of $(\HH,\dr)$ is obtained just by dropping the vertical component $z$ of the argument, and then applying to $v$ the usual Euclidean Busemann function in the direction which is opposite to the limit direction of the $v_n$'s, rotated by an angle $\vartheta$ depending on the quadratic rate of divergence of the sequence ($\vartheta$ is zero for points diverging sub-quadratically, and $\vartheta=\pm \pi$ when the divergence is sup-quadratical).

\vspace{2mm}
  These properties mark a remarkable  difference  with the theory of nonpositively curved, simply connected spaces.

\appendix
\section{Length-minimizing curves for \texorpdfstring{$\dcc$ and $\dr$}{dcc and dr} }\label{sec23031017}
In the Heisenberg group, locally length-minimizing curves are smooth solutions of an Hamiltonian system both in the Riemannian and in the subRiemannian case.
Locally length-minimizing curves are also called \emph{geodesics}.

Let $\dr$ be a Riemannian metric on $\HH$ with horizontal space $(\hh,g)$.
Let $V\subset\hh$ be the plane orthogonal to $[\hh,\hh]$ and let $\dcc$ be the strictly subRiemannian metric on $\HH$ with horizontal space $(V,g|_V)$.

Fix a basis $(X,Y,Z)$ for $\hh$ such that $(X,Y)$ is an orthonormal basis of $(V,g|_V)$ and $Z=[X,Y]$.
Set $\zeta = \sqrt{g(Z,Z)}$.

Let $\dcc$ be the strictly subRiemannian metric on $\HH$ with horizontal space $(V,g|_V)$.

The basis $(X,Y,Z)$ induces the exponential coordinates $(x,y,z)$ on $\HH$, i.e.,~$(x,y,z)=\exp(xX+yY+zZ)$.
We will work in this coordinate system.

The Riemannian and subRiemannian length-minimizing curves are known and we recall their parametrization in the following two propositions.

\begin{Prop}[subRiemannian geodesics]\label{teoGeoDc}
 	All the non-constant locally length-minimizing curves of $\dcc$ starting from $0$ and parametrized by arc-length are the following:
	given $k\in\R\setminus\{0\}$ and $\theta\in\R$
	\begin{enumerate}[label=(\alph*)]
	\item[\textsc{(Type I)}] 	The horizontal lines $t\mapsto (t\cos\theta , t\sin\theta, 0)$; 
	\item[\textsc{(Type II)}]	The curves $t\mapsto (x(t),y(t),z(t))$ given by
		\[
		\left\{
		\begin{aligned}
		 	x(t) &= \frac1k\left( \cos\theta(\cos(kt)-1) - \sin\theta \sin(kt) \right)	\\
			y(t) &= \frac1k\left( \sin\theta(\cos(kt)-1) + \cos\theta \sin(kt) \right)	\\
			z(t) &= \frac1{2k} t -\frac1{2k^2}\sin(kt) 
		\end{aligned}
		\right.
		\]
		Here the derivative at $t=0$ is $(-\sin\theta,\cos\theta,0)$.
	\end{enumerate}
\end{Prop}
\begin{Prop}[Riemannian geodesics]\label{teoGeoDr}
 	All non-constant locally length-minimizing curves of $\dr$ parametrized by a multiple of arc-length and starting from $0$ are the following:
	given $k\in\R\setminus\{0\}$ and $\theta\in\R$
	\begin{enumerate}[label=(\alph*)]
	\item[\textsc{(Type 0)}] 	The vertical line $t\mapsto (0,0,t)$;
	\item[\textsc{(Type I)}] 	The horizontal lines $t\mapsto (t\cos\theta , t\sin\theta, 0)$; 
	\item[\textsc{(Type II)}] 	The curves $t\mapsto (x(t),y(t),z(t))$ given by
		\[
		\left\{
		\begin{aligned}
		 	x(t) &= \frac1k\left( \cos\theta(\cos(kt)-1) - \sin\theta \sin(kt) \right)	\\
			y(t) &= \frac1k\left( \sin\theta(\cos(kt)-1) + \cos\theta \sin(kt) \right)	\\
			z(t) &= \frac1{2k} t -\frac1{2k^2}\sin(kt) + \frac{k}{\zeta^2}t
		\end{aligned}
		\right.
		\]
		Here the derivative at $t=0$ is $(-\sin\theta,\cos\theta,\frac k{\zeta^2})$, which has Riemannian length $\sqrt{1+\frac{k^2}{\zeta^2}}$.
	\end{enumerate}
\end{Prop}

The expression of geodesics helps us to prove the following facts.

\begin{Cor}\label{corGeoDrtypeI}
The horizontal lines of  \textsc{Type I} are globally  $\dr$- and  $\dcc$-length-minimizing curves.
\end{Cor}

\begin{Cor}\label{corGeoDr1}
 	Both $\dr$- and locally $\dcc$-length-minimizing curves $\gamma$ of \textsc{Type II} are not minimizing from $0$ to $\gamma(t)$ if $|t| > \frac{2\pi}{k}$.
\end{Cor}
\begin{proof}
 	This statement depends on the fact that, if we fix $k\in\R\setminus\{0\}$, then for all $\theta$ the corresponding length-minimizing curves $\gamma_{k,\theta}$ of \textsc{Type II} meet each other at the point $\gamma_{k,\theta}(2\pi/k) = (0,0,\frac{2\pi}{k^2})$ or $\gamma_{k,\theta}(2\pi/k) = (0,0,\frac{2\pi}{k^2}+\frac{2\pi}{\zeta^2})$.
\end{proof}
\begin{Cor}\label{corGeoDr1bis}
	The locally $\dr$-length-minimizing curve $\gamma$ of \textsc{Type 0}, $t\mapsto(0,0,t)$, is not minimizing from $0$ to $\gamma(t)$ for $|t|>\frac{2\pi}{\zeta}$.
\end{Cor}
\begin{proof}
	For $k>0$ let $\gamma_k$ be the $\dr$-length-minimizing curve of \textsc{Type II} with this $k$ and $\theta=0$. 
	Then $(\gamma_k)_3(\frac{2\pi}{k}) = \frac{\pi}{k^2} + \frac{2\pi}{\zeta^2}$.
	Letting $k\to\infty$ we obtain $\hat z:= \frac{2\pi}{\zeta^2}$. 
	This means that for every $\epsilon>0$ there is $ z \le \hat z+\epsilon$ and $k>0$ such that $\gamma_k( \frac{2\pi}{k}) = (0,0,z)$. 
	Therefore $t\mapsto (0,0,t)$ cannot be minimizing after $z$, and therefore after $\hat z$.
\end{proof}
\begin{Cor}\label{corGeoDr2}
 	If $p=(x,y,p_3)$ and $q=(x,y,q_3)$, then
	\[
	\dcc(p,q) = 2\sqrt\pi\cdot\sqrt{|p_3-q_3|} .
	\]
\end{Cor}
\begin{proof}
 	First suppose $p=0$: we have to prove that $\dcc(0,(0,0,z))= 2\sqrt\pi \sqrt{|z|}$. 
	This is done by looking at the length-minimizing curves: they comes from complete circle of perimeter $2\pi R=d$ and area $\pi R^2=|z|$, so that $\dcc(0,(0,0,z))=2\pi\sqrt{\frac{|z|}{\pi}} = 2\sqrt\pi\sqrt{|z|}$.
	
	The general case follows from the left-invariance of $\dcc$: 
	\begin{align*}
	\dcc((x,y,p_3),(x,y,q_3)) 
	&= \dcc (0,(x,y,p_3)^{-1}(x,y,q_3)) \\
	&=\dcc(0,(0,0,q_3-p_3)) .
	\end{align*}
\end{proof}
\begin{Cor}\label{cor1103}
 	If $p\in\{z=0\}$, then
	\[
	\dcc(0,p) = \dr(0,p)
	\]
\end{Cor}
\begin{Cor}\label{corGeoDr}
	$\dr$- and $\dcc$-length-minimizing curves of \textsc{Type II} are in bijection with the following rule: If $\eta:[0,T]\to\HH$ is a $\dcc$-length-minimizing curve of \textsc{Type II}, then
	\[
	\gamma(t) = \eta(t) + \left(0 ,0, \frac{kt}{\zeta^2} \right)
	\]
	is $\dr$-length-minimizing of \textsc{Type II}, where $k\in\R$ is given by $\eta$.
	Moreover, it holds
	\[
	\|\mc(\gamma')\|^2 = 1 + \frac{k^2}{\zeta^2}
	\]
	and 
	\[
	\dcc(\gamma(t),\eta(t)) = 2\sqrt{\pi} \sqrt{\frac{kt}{\zeta^2} }
	\]
\end{Cor}
\begin{proof}
 	All the statements come directly from the expression of the geodesics.
	Notice that a $\dcc$-length-minimizing curve $\eta$ of \textsc{Type II} is parametrized by arc-length, i.e., $\|\mc(\eta')\|\equiv1$. 
	
	On the other hand, the corresponding $\dr$-length-minimizing curve $\gamma$ has derivative $\mc(\gamma')=\mc(\eta') + \frac k{\zeta^2} Z$, where $\mc(\eta')$ is orthogonal to $Z$.
	Hence $\|\mc(\gamma')\|^2 = 1 + \frac{k^2}{\zeta^2}$
\end{proof}

\printbibliography

\end{document}